\documentclass[11pt,a4paper]{amsart}
\usepackage{amsopn, amssymb, amscd, amsmath, amsfonts}
\usepackage{verbatim}
\usepackage[english]{babel}

\newcommand{\mc}{\mathfrak c}

\newcommand{\mg}{\mathfrak g}
\newcommand{\mh}{\mathfrak h}

\newcommand{\mn}{\mathfrak n}

\newcommand{\mz}{\mathfrak z}
\newcommand{\mso}{\mathfrak{so}}

\newcommand{\msp}{\mathfrak{sp}}
\newcommand{\msl}{\mathfrak{sl}}

\newcommand{\HH}{\mathbb H}
\newcommand{\RR}{\mathbb R}
\newcommand{\CC}{\mathbb C}
\newcommand{\OO}{\mathbb O}

\theoremstyle{plain}
\newtheorem{Theorem}{Theorem}

\newtheorem{prop}{Proposition}
\newtheorem{cor}{Corollary}

\begin{document}

\title[Parabolic nilradicals of Heisenberg type]{Parabolic nilradicals of Heisenberg type}
\author{Aroldo Kaplan}
\author{Mauro Subils}
\date{\today}
\address{A. Kaplan: CIEM-CONICET, Famaf, U.N.C., Cordoba 5000, Argentina, and Department of Mathematics, U. of Massachusetts, Amherst, MA 01002, USA}
\email{kaplan@math.umass.edu}
\address{M. Subils: CONICET-FCEIA, U.N.R., Pellegrini 250, 2000 Rosario, Argentina.} \email{msubils@gmail.com}

\thanks{This work was supported by Secyt-UNR, Scyt-UNC and CONICET
}

\maketitle


\maketitle

 \begin{abstract}
We show that every non-compact simple real Lie algebra not isomorphic to $\mso(n,1)$ has a unique conjugacy class of parabolic subalgebras whose nilradical is of Heisenberg type, or non-singular, and give some applications.
\end{abstract}

\section{Introduction}

We show that every non-compact simple real Lie algebra not isomorphic to $\mso(n,1)$ has a unique conjugacy class of parabolic subalgebras whose nilradical is of Heisenberg type, or non-singular. The nilradicals that appear are all  of the form
 $\mathbb F^{2n}\oplus \mathbb F$ or
  $\mathbb F^{n}\oplus \Im(\mathbb F)$
where $\mathbb F = \RR, \CC, \HH, \OO$. The algebras $\mso(n,1)$ have a unique conjugacy class of parabolic subalgebras, all with abelian nilradicals, and are the unique simple algebras with these properties.

There is an extensive literature dealing with special pairs $(G,P)$ with $G$ a simple Lie group and $P\subset G$ a parabolic subgroup with abelian or Heisenberg-like nilradical \cite{CSl}. In the abelian case the spaces $G/P$ are called Almost Hermitian Symmetric. Examples with Heisenberg type nilradical are the parabolics associated to contact structures over the division algebras. Heisenberg type groups are square-integrable in the sense of Wolf \cite{W}, but the converse is not true. The existence and uniquennes properties seem characteristic of the type dicussed here.

Consequences for the automorphism group and the Tanaka prolongation of general algebras of Heisenberg type are deduced.
Details of the proofs and further consequences will be discussed in a forthcoming paper.

P. Deligne pointed out an error in our original statement. J. Huerta noted that the analogous of Theorem 1 for complex goups (which follows from it) was known in complex contact geometry.  E. Hullet directed the doctoral disertation of one of us (Subils), which contained the results on Tanaka prolongation. J. Wolf led us to his square-integrable nilradicals, which include properly those considered here. We thank all of them, as well as J. Baez, J. Humphreys and F. Ricci, for their help and advice.

\section{Main Theorem}


The real division algebras $\mathbb F =$ $\mathbb R,\CC,\HH,\OO$, give rise to two kinds of 2-step real nilpotent Lie algebras:
\begin{align}\label{def:hn}
    \mathfrak h_n(\mathbb F)=\mathbb F^{2n}\oplus \mathbb F \\
     [(a,b),(c,d)]=a^{t}d-c^{t}b\nonumber,
\end{align}
for $a,\,b,\,c,\,d\in \mathbb F^{n}$ and for any $n\geq 1$ if $\mathbb F=\mathbb R, \mathbb C, \mathbb H$, and for $n=1$ if $\mathbb F=\mathbb O$.
\begin{align}\label{def:hn'}
    \mathfrak h'_{p,q}(\mathbb F)=\mathbb F^{p+q}\oplus \Im(\mathbb F) \\
     [(a,b),(c,d)]=a^{t}\overline{c}-c^{t}\overline{a}+\overline{b}^{t}d-\overline{d}^{t}b\nonumber,
\end{align}
for $a,\,c\in \mathbb F^{p}$, $b,\,d\in \mathbb F^{q}$ and for any $p,q\geq 1$ if $\mathbb F= \mathbb C, \mathbb H$, and for $p=1$ and $q=0$ if
$\mathbb F=\mathbb O$.

\begin{Theorem}\label{theorem:niltipoH}
Every real simple non-compact Lie algebra not isomorphic to $\mathfrak {so}(n,1)$ has a unique conjugacy class of parabolic subalgebras whose nilradical is isomorphic to one of the $\mathfrak h_n(\mathbb F)$, $\mathfrak h_{p,q}'(\mathbb F)$ with $\mathbb F=\mathbb C,\mathbb H,\mathbb O$.  On the other hand, $\mso(n,1)$ has a unique conjugacy class of parabolic subalgebras, all with abelian nilradicals, and it is the unique simple algebra with these properties.
\end{Theorem}

\begin{proof}
Every non-compact real form of a complex simple Lie algebra, except for $\msl(n+1,\HH)$, $\mso(n,1)$, $\msp(p,q)$, $EIV$ and $FII$, has a real contact grading (p. 312, \cite{CSl}). This is equivalent to have a parabolic subalgebra with nilradical isomorphic to $\mh'_{n}(\CC)$, the real Heisenberg algebra. Therefore every simple complex Lie algebra except $\msl(2,\CC)$ has a parabolic subalgebra with nilradical isomorphic to $\mh_{n}(\CC)$. Note that $\msl(2,\CC) \cong \mso(3,1)$ as real algebras.

Recall that  $\msp(1,1)\cong \mso(4,1)$ and $\msl(2,\HH)\cong \mso(5,1)$.
For $\msl(n+1,\HH)$ with $n\geq 2$ take the grading determined by $\{\alpha_{2}, \alpha_{2n}\}$; the corresponding parabolic subalgebra has nilradical  $\mh_{n-1}(\HH)$.
For $\msp(p+1,q+1)$, the parabolic subalgebra determined by $\{\alpha_{2}\}$ has nilradical isomorphic to $\mh_{p,q}'(\HH)$.
For $EIV$ and $FII$ the nilpotent Iwasawa subalgebras (nilradical of the Borel subalgebra) are isomorphic to $\mh_{1}(\OO)$ and $\mh_{1,0}'(\OO)$, respectively.

To show uniqueness let $\Gamma$ be the system of restricted roots of a simple real Lie algebra, $\Gamma^{+}$ a choice of positive system and  $\Gamma^{0}$ the set of simple restricted roots.  Analogously to the complex case, the conjugacy classes of parabolic subalgebras are in one to one correspondence with a subset $\Phi\subset\Gamma^{0}$. It is not hard to see that when the nilradical of the parabolic is two step, the maximal restricted root $\gamma$ must have height two with respect to $\Phi$, and when it is non singular, i.e. $ad\,X : \mg \to \mz$ is onto for every $X\in\mg\setminus\mz$, for every restricted root $\alpha$ of height $1$, $\gamma- \alpha$ must be a root of height $1$. Looking at every non reduced root system we get a unique possible subset $\Phi$ with these properties up to equivalence, except for $A_{1}$ where there is none.

The last statement follows from the observation that $\mso(n,1)$ is the only simple algebra  whose system of restricted roots is $A_{1}$ and has only one positive restricted root.
\end{proof}

\section{Some consequences for non-singular and H-type algebras}

Recall that a 2-step graded nilpotent real Lie algebra
 $\mathfrak{n= n^{-1} \oplus n^{-2}}$  is of Heisenberg (or H) type if there is a graded positive inner product such that $\mathfrak n^{-1}$ is a non-trivial real unitary module over the Clifford algebra $C(\mathfrak n^{-2})$ and the bracket is given by
 $$<[x,y],z>_{\mathfrak n^{-2}} = < z\cdot x,y>_{\mathfrak n^{-1}}$$
 (\ref{def:hn}) and (\ref{def:hn'}) above are examples.

 Let $J_{z}:\mathfrak n^{-1}\to \mathfrak n^{-1}$ be the transformation $J_{z}x=z\cdot x$

\begin{Theorem}\label{isomectricisomorf}
	If two Lie algebras of H-type are isomorphic then there exists an isometric isomorphism between them.
	In particular, the inner product on the center is unique up to multiple.
\end{Theorem}

\begin{proof}
	We can assume that we have a Lie algebra of H-type $\mg$ with two different inner products $(\ ,\,)$ and $\langle\ ,\,\rangle$.
	There exists a positive definite matrix $P$ symmetric with respect to both inner products such that $\langle x ,y\rangle = (Px , Py)$  for all $x, y\in\mn$. We will show that the mapping defined by $P:\mn\rightarrow\mn$ is an automorphism of $\mn$.
	
	For $z\in\mn^{-2}$, let $J_{z}$ and $K_{z}$ the transformations defined from $(\ ,\,)$ and $\langle\ ,\,\rangle$, respectively. Then
	\begin{align*}
		(P^{2}K_{z}x,y) = \langle K_{z}x, y\rangle = \langle z,[x, y]\rangle=(P^{2}z,[x,y])=(J_{P^{2}z}x,y)
	\end{align*}
	for $x,y\in\mn^{-1}$ and $z\in\mn^{-2}$. We conclude that
	\begin{equation*}\label{eq1propunicidadpi}
		P^{2}K_{z}=J_{P^{2}z}.
	\end{equation*}
	Squaring both sides and multiplying by $K_{z}$ on the right we obtain
	\begin{equation}\label{eq2propunicidadpi}
		P^{2}K_{z}P^{2}= \frac{(P^{2}z , P^{2}z)}{(Pz , Pz)}K_{z}.
	\end{equation}
	$P$ diagonalize so there exists $\{z_{i}\}_{i=1}^{n}$ base of $\mn^{-2}$ such that $Pz_{i}=\lambda_{i}z_{i}$ with $\lambda_{i}>0$ for $i=1,\ldots,n$. Let $z_{0}=\sum_{i=1}^{n} z_{i}$,
	\begin{equation*}
		\sum_{i=1}^{n}\lambda^{2} K_{z_{i}}=\lambda^{2} K_{z_{0}}=P^{2}K_{z_{0}}P^{2}=\sum_{i=1}^{n}P^{2}K_{z_{i}}P^{2}=\sum_{i=1}^{n}\lambda_{i}^{2}K_{z_{i}}
	\end{equation*}
	where $\lambda^{2}=\frac{(P^{2}z_{0} , P^{2}z_{0})}{(P z_{0} , P z_{0})}$.
	Then $\lambda_{i}=\lambda$ for all $i=1,\ldots,n$ and $P|_{\mn^{-2}}=\lambda Id$.
	Rewriting (\ref{eq2propunicidadpi}) as
	\begin{equation*}
		P^{2}K_{z}P^{2}= \lambda^{2} K_{z},
	\end{equation*}
	we see that $K_{z}$ interchanges the eigenspaces associated to the eigenvalues $\mu^{2}$ and $\lambda^{2}/\mu^{2}$ of $P^{2}$,
	which also are the eigenspaces of $P$ corresponding to the eigenvalues $\mu$ and $\lambda/\mu$, respectively. So,
	\begin{equation*}
		P K_{z}P= \lambda K_{z}=K_{Pz}.
	\end{equation*}
	Since $P$ is symmetric, we conclude that $P$ is an automorphism of $\mn$.
\end{proof}

The last statement was first proved in [KT].

\begin{Theorem}\label{theorem:irredmg}
	The orthogonal automorphisms of a Lie algebra of H-type act irreducibly on $\mathfrak{n}^{-1}$.
\end{Theorem}

\begin{proof}
	Every Lie algebra of H-type has automorphisms of the form:
	
	$$\left(
	\begin{array}{cc}
	J_{z} &  \\
	& -r_{z} \\
	\end{array}
	\right)$$
	for every $z\in\mathfrak{n}^{-2}$. 
	Then, if the Lie algebra is irreducible, the action of this automorphism on $\mathfrak{n}^{-1}$ is irreducible.
	
	If the Lie algebra is not irreducible, consider $\mathfrak{v}_{1}$ and $\mathfrak{v}_{2}$ two irreducible sub-representations of $\mathfrak{n}^{-1}$, then $\mathfrak{v}_{1}\oplus\mathfrak{n}^{-2}$ and $\mathfrak{v}_{2}\oplus\mathfrak{n}^{-2}$ are irreducible H-type subalgebras so they are isomorphic. By Theorem \ref{isomectricisomorf} there exists an isometric isomorphism $\theta:\mathfrak{v}_{1}\oplus\mathfrak{n}^{-2}\rightarrow\mathfrak{v}_{2}\oplus\mathfrak{n}^{-2}$. Define the orthogonal automorphism $\Theta:\mn\rightarrow\mn$ by $\left.\Theta\right|_{\mathfrak{n}^{-2}}=\left.\theta\right|_{\mathfrak{n}^{-2}}$, $\left.\Theta\right|_{\mathfrak{v}_{1}}=\left.\theta\right|_{\mathfrak{v}_{1}}$, $\left.\Theta\right|_{\mathfrak{v}_{2}}=\left.\theta\right|_{\mathfrak{v}_{1}}^{-1}$ and $\left.\Theta\right|_{\mc}=Id$ where $\mc$ is the orthogonal complement of $\mathfrak{v}_{1}\oplus\mathfrak{v}_{2}\oplus\mathfrak{n}^{-2}$. Since $\mathfrak{v}_{1}$ and $\mathfrak{v}_{2}$ are arbitrary the prove is complete.
\end{proof}

\begin{prop}[Ottazzi, Warhurst \cite{OW}]\label{prop:tipoHfinito}
An H-type algebra with center of dimension greater than $2$ is of finite type, i.e. has finite dimensional Tanaka prolongation.
\end{prop}

\begin{cor}
For a Lie algebra of H-type the following conditions are equivalent:
\begin{enumerate}
  \item to be the nilradical of a parabolic subalgebra of a simple Lie algebra;
  \item to have non-trivial Tanaka prolongation;
  \item to be isomorphic to one of the Lie algebras $\mh_{n}(\CC)$, $\mh'_{n}(\CC)$, $\mh_{n}(\HH)$, $\mh'_{p,q}(\HH)$, $\mh_{1}(\OO)$ and  $\mh'_{1,0}(\OO)$.
\end{enumerate}
\end{cor}

\begin{proof}
It is a consequence of the proof of Theorem \ref{theorem:niltipoH}, Theorem \ref{theorem:irredmg}, Proposition \ref{prop:tipoHfinito}, Lemma 5.8 in \cite{Y} and the fact that $\mh'_{n}(\CC)$ and $\mh_{n}(\CC)$ have infinite Tanaka prolongation.
\end{proof}




\end{document}